\documentclass[preprint,12pt]{elsarticle}

\usepackage{tikz}
\usepackage{amsmath, amsthm}
\usepackage{amsfonts}
\usepackage{amssymb}
\usepackage{mathrsfs}
\def\al{\alpha}
\def\be{\beta}
\def\ga{\gamma}
\def\de{\delta}

\def\Om{\Omega}

\def\Im{\operatorname{Im}}

\newcommand{\rep}{\operatorname{rep}}

\newcommand{\bbR}{{\mathbb R}}

\def\dsm#1,#2..#3{\bigoplus_{{#1}={#2}}^{#3}}

\newcommand{\id}{1\kern-.25em{\text{{\rm l}}}} 

\newcommand{\isoto}{\ \raise.8ex\hbox{$^{\sim}$}\kern-.7em\hbox{$\to$}\ } 

\newcommand{\bg}{%
\family{cmr}\size{20}{12pt}\selectfont}

\newcommand{\bigzerou}{%
\smash{\lower1.7ex\hbox{\bg 0}}}

\renewcommand{\k}{\Bbbk}

\newtheorem{definition}{Definition}[section] 

\newtheorem{thm}[definition]{Theorem} 

\newtheorem{lem}[definition]{Lemma} 

\newtheorem{prop}[definition]{Proposition} 

\usepackage{amscd}
\usepackage[all]{xy}
\usepackage{tikz-cd}

\usepackage{amssymb}

\makeatletter
\def\ps@pprintTitle{%
\let\@oddhead\@empty
\let\@evenhead\@empty
\let\@oddfoot\@empty
\let\@evenfoot\@oddfoot
}
\makeatother

\makeatletter
\def\ps@pprintTitle{%
\let\@oddhead\@empty
\let\@evenhead\@empty
\def\@oddfoot{\reset@font\hfil\thepage\hfil}
\let\@evenfoot\@oddfoot
}
\makeatother

\begin{document}

\begin{frontmatter}



\title{Double Flag Varieties and Representations of Quivers}


\author{Hiroki Homma}

\address{$homma.hiroki.249$@$s.kyushu$-$u.ac.jp$}

\begin{abstract}

We gave a classification of $P$ and $Q$ with a finite number of $K$-orbits of a double flag variety $G/P\times K/Q$ 
for a symmetric pair $(G, K)$ when $G=GL_{m+n}$ and $K=GL_{m}\times GL_{n}$, 
and a description of $K$-orbits when the number of $K$-orbits of $G/P\times K/Q$ is finite. 
We solved the problem by providing a correspondence between the $K$-orbits and the quiver representations.
\end{abstract}


\end{frontmatter}


\section{Introduction}
For a reductive group $G$ and its symmetric subgroup $K$, 
the direct product $G/P \times K/Q$ of partial flag varieties $G/P$ and $K/Q$ is called a double-flag variety for a symmetric pair $(G, K)$, 
and the diagonal action of $K$ on $G/P \times K/Q$ is an important object applied to branching rules of representations, e.g. \cite{km}. 
In particular, two problems are as follows: \\
$(1)$What are the pairs $G$, $K$, $P$, and $Q$ such that there are only finitely many $K$-orbits on $G/P \times K/Q$ ?\\
$(2)$Can we describe the $K$-orbits on $G/P \times K/Q$ when there are only finitely many $K$-orbits ?\\
We solved this problem by corresponding the $K$-orbits to the quiver representations.

The problem of the finiteness of the number of $G$-orbits of a triple-flag variety in the case of a group $G$ was initiated by Magyar-Weyman-Zelevinsky\cite{mwz1999}, \cite{mwz2000}, 
then they gave the finiteness determination condition for $G$-orbits and the description of their $G$-orbits in the case of G=GL, Sp. 
Later, He-Nishiyama-Ochiai-Oshima\cite{hnoo} generalized the setting to double-flag varieties $G/P \times K/Q$ for symmetric pairs, 
and gave a classification of $P$ and $Q$ such that the number of $K$-orbits is finite in the special case where $P$ or $Q$ is a Borel subgroup. 
However, many problems remain unsolved, such as the complete classification of $P$ and $Q$ such that the number of $K$-orbits is finite when $P$ and $Q$ are arbitrary, 
and the description of the orbit decomposition. 

In this paper, 
we classified $P$ and $Q$ such that the number of $K$-orbits is finite for $G=GL_{m+n}$ and $K=GL_{m}\times GL_{n}$, and described the $K$-orbits in these cases. 
Also, while previous research used the root system to study the problem, in this paper, we used the method of the representation theory of quivers to solve the problem.

The joint flag variety introduced in this paper is a dense open sub-variety of the triple-flag variety with the action of $G$. 
For some of them, the number of $G$-orbits in the joint flag variety is finite, while the number of $G$-orbits in the triple flag variety is infinite. 
The construction of this joint-flag variety also uses ideas from the representation theory of quiver. 
This discovery was made possible by research that spanned both fields.

In this paper, the coefficient field $\k$ is an algebraically closed field with characteristic zero. 

In Section 2, we introduce a joint-flag variety ${\rm Jl}_{\bf d}(V)$ into the triple-flag variety ${\rm Fl}_{\bf d}(V)$ of $GL(V)$. We now consider the $GL(V)$-action on ${\rm Jl}_{\bf d}(V)$. In this section, we give the condition for ${\rm Jl}_{\bf d}(V)$ to have a finite number of $GL(V)$-orbits. Furthermore, we describe the $GL(V)$-orbits when the number of $GL(V)$-orbits is finite.

In Section 3, we give a correspondence between the $GL(V_1)$-orbits of the joint flag variety ${\rm Jl}_{\bf d}(V_1)$ and the $GL(V_2)\times GL(V_3)$-orbits of the double flag variety ${\rm Dl}_{\bf d}(V_2\bigoplus V_3)$ for a symmetric pair ($GL(V_1)$, $GL(V_2)\times GL(V_3)$). Then, from this correspondence and the results of Section 2, we obtain the following two results.\\
$(1)$ The decision condition of ${\rm Dl}_{\bf d}(V_2\bigoplus V_3)$ for the number of $GL(V_2)\times GL(V_3)$-orbits of ${\rm Dl}_{\bf d}(V_2\bigoplus V_3)$ to be finite is given.\\
$(2)$ Describe the K-orbits of ${\rm Dl}_{\bf d}(V_2\bigoplus V_3)$ when the number of $GL(V_2)\times GL(V_3)$-orbits of ${\rm Dl}_{\bf d}(V_2\bigoplus V_3)$ is finite.

\section{Joint Flag Varieties}
Let an integer partition that allows ${\bf a}=(a_1, \cdots ,a_p)$ to contain a zero part. 
We write $|{\bf a}|:=a_1+\cdots +a_p$, $||{\bf a}||^2:={a_1}^2+\cdots +{a_p}^2$, ${\bf a'}:=(a_1, \cdots , a_{p-1})$, 
$\ell ({\bf a}):=p$ called the length of {\bf a}, and $(a^p):=\underbrace{(a, \cdots ,a)}_{p\; \text{parts}}$. 
We denote by ${\rm Fl}_{\bf a}(V)$ the flag variety consisting of all 
flags $(0=A_0\subset A_1\subset \cdots \subset A_{p-1}\subset A_p=V)$ such that 
\[
\dim{A_i}-\dim{A_{i-1}}=a_i \, \, \, \, \, (i=1, \cdots , p).
\]

For any positive integer $p, q, r$, let $Q_{r, p, q}$ be a quiver of the following form:
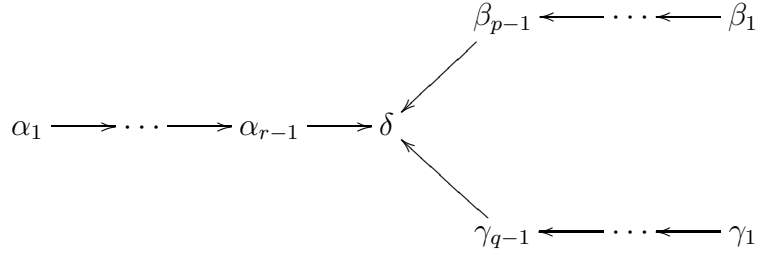
\[
\xymatrix{
 & & & & \be_{p-1} \ar[dl] & \cdots \ar[l] & \be_1 \ar[l] \\
 \al_1  \ar[r]  & \cdots \ar[r] & \al_{r-1} \ar[r] & \de  \\
 & & & & \ga_{q-1} \ar[ul] & \cdots \ar[l] & \ga_1 \ar[l]  
}
\]
and $\rep{\k Q_{r, p, q}}$ be a category of finite-dimensional representations of $Q_{r, p, q}$. 

\begin{definition}
Define the {\bf Joint flag category} ${\mathscr{J}}_{r, p, q}$ as follows$:$ \\
${\mathscr{J}}_{r, p, q}$ is a full sub-category of $\rep{\k Q_{r, p, q}}$, whose objects are given in the following forms of $(V, A, B, C)\in \rep{\k Q_{r, p, q}}$$:$
\\
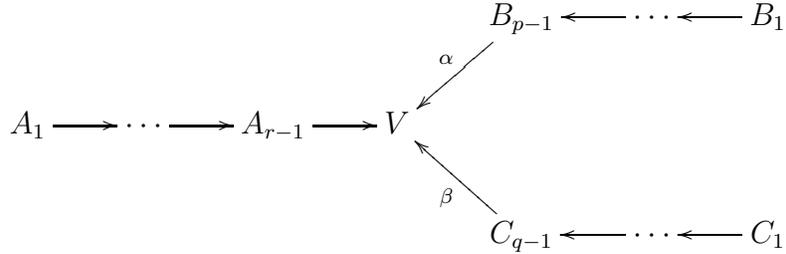
\[
\xymatrix{
 & & & & B_{p-1} \ar[dl]_\alpha & \cdots \ar[l] & B_1 \ar[l] \\
 A_1  \ar[r]  & \cdots \ar[r] & A_{r-1} \ar[r] & V  \\
 & & & & C_{q-1} \ar[ul]^\beta & \cdots \ar[l] & C_1 \ar[l]
}
\]
 where all arrows are injections, and $\Im{\alpha}\bigoplus\Im{\beta}=V$. 
\end{definition}

For any positive integer $p, q, r$, let ${{\Lambda}_{r, p, q}}^{J}$ denote  the additive semi-group of all triples of partitions $({\bf a, b, c})$ such that 
$(\ell ({\bf a}), \ell ({\bf b}), \ell ({\bf c}))=(r, p, q)$, and $|{\bf a}|=|{\bf b'}|+|{\bf c'}|$. 
When there is no risk of ambiguity, we drop the subscripts and write ${\Lambda}^{J}, {\mathscr{J}}$. 

Then by embedding ${\mathscr{J}}_{r, p, q}$ in $\rep{\k Q_{r, p, q}}$, 
the object of ${\mathscr{J}}_{r, p, q}$ is directly decomposed on $\rep{\k Q_{r, p, q}}$ except for the order and the isomorphism, 
by applying Krull-Schmidt Theorem\cite{krull}. 
It's also easy to check. The fact that this decomposition is on ${\mathscr{J}}_{r, p, q}$.

\begin{definition}
For any ${\bf d}\in {\Lambda}^{J}$, we define ${\rm Jl}_{\bf d}(V)$ as a sub-variety of ${\rm Fl}_{\bf d}(V)$ consisting of the whole of $(A, B, C)\in {\rm Fl}_{\bf d}(V)$ satisfying $(V, A, B, C)\in \mathscr{J}$. 
At this time ${\rm Jl}_{\bf d}(V)$ is called Joint flag variety.
\end{definition}

\begin{definition}
Let us say that $F\in {\mathscr{J}}$ is represented by ${\bf d}\in {\Lambda}^{J}$ when there exists a $(A, B, C)\in {\rm Jl}_{\bf d}(V)$ such that $F\cong (V, A, B, C)$. 
In this case, we denote $F$ by $F_{\bf d}$.
\end{definition} 

Then the following are equivalent for $(A, B, C), (A', B', C')\in {\rm Jl}_{\bf d}(V)$: 
\begin{align*}
&(1)\; GL(V)(A, B, C)=GL(V)(A', B', C') \\
&(2)\; (V, A, B, C)\cong (V, A', B', C')\; on\; {\mathscr{J}}
\end{align*}

Therefore we get the following:\\
(1) $\# \{ GL(V)(A, B, C)\, |\, (A, B, C)\in {\rm Jl}_{\bf d}(V)\} =\# \{ $isomorphism classes of $F_{\bf d}\in {\mathscr{J}}\} $.\\
(2) A concrete description of the orbit decomposition of ${\rm Jl}_{\bf d}(V)$ can be given by finding all the isomorphism classes of $F_{\bf d}$. \\

Now we can translate the problem of the number and description of $GL(V)$-orbits of ${\rm Jl}_{\bf d}(V)$ into the problem of finding the isomorphism classes of the representation of the quiver.

\begin{definition}
The {\bf Tits quadratic form} is define to be the form 
\[
Q{\bf (a, b, c)}=\dim{GL(V)}-\dim{{\rm Fl}_{\bf a}(V)}-\dim{{\rm Fl}_{\bf b}(V)}-\dim{{\rm Fl}_{\bf c}(V)}, 
\]
where ${\bf (a, b, c)}\in {\Lambda}^{J}$. \\
An easy calculation shows that 
\[
Q{\bf (a, b, c)}=(||{\bf a}||^2+||{\bf b}||^2+||{\bf c}||^2-(\dim{V})^2)/2. 
\]
\end{definition}

We get the following proposition immediately from (\cite{mwz1999}, Proposition 3.1.).

\begin{prop} \label{sec:2.1} 
Suppose ${\bf d}\in {\Lambda}^{J}$ is the dimension vector of an indecomposable object of $\mathscr{F}$ with $Q({\bf d})\geq 1$.
Then $Q({\bf d})=1$, and there is a unique isomorphism class ${\cal{I}}_{\bf d}$ of indecomposable objects with the 
dimension vector {\bf d}.
\end{prop}

\begin{lem}
$\dim{{\rm Fl}_{\bf d}(V)}=\dim{{\rm Jl}_{\bf d}(V)}$
\end{lem}
\begin{proof}
The proposition is clear from the following. ${\rm Fl}_{\bf d}(V)$ is connected and ${\rm Jl}_{\bf d}(V)$ is an open set of ${\rm Fl}_{\bf d}(V)$.
\end{proof}

\begin{definition}
A dimension vector ${\bf d=(a, b, c)}\in {\Lambda}^{J}$ is called finite type in $\mathscr{J}$ if 
the number of ${GL}_{|{\bf a}|}$-orbits in ${\rm Jl}_{\bf d}(V)$ is finite.
\end{definition}

\begin{definition}
A dimension vector ${\bf d=(a, b, c)}\in {\Lambda}^{J}$ is called finite type in $\mathscr{F}$ if 
the number of ${GL}_{|{\bf a}|}$-orbits in ${\rm Fl}_{\bf d}(V)$ is finite.
\end{definition}

\begin{prop}\label{sec:3.5}
If {\bf d} is a dimension vector of finite type in $\mathscr{J}$ and satisfies $Q({\bf d})=1$ then there exists a Schur indecomposable 
${I}_{\bf d}$ with the dimension vector {\bf d}.
\end{prop}
\begin{proof}
Since ${\bf d}$ is of finite type, ${\rm Jl}_{\bf d}(V)$ has a Zariski open orbit $\Om$.
So take the isomorphism class ${I}_{\bf d}$ of ${\rm Jl}_{\bf d}(V)$ corresponding to $\Om$.
Then take any quiver representative $F$ corresponding to ${I}_{\bf d}$.
Then we get the following equation.

\begin{equation*}
\langle {I}_{\bf d}, {I}_{\bf d} \rangle =\dim{{Stab}_{GL(V)}(F)}=\dim{GL(V)}-\dim{{\rm Jl}_{\bf d}(V)}=Q({\bf d})=1.
\end{equation*}

Hence ${I}_{\bf d}$ becomes a Schur indecomposable.
\end{proof}

\begin{definition}
We say that a non-zero dimension vector ${\bf d'}\in {\Lambda}^{J}$ is a summand of ${\bf d}\in {\Lambda}^{J}$ on $\mathscr{J}$ 
if ${\bf d-d'}\in {\Lambda}^{J}$. 
\end{definition}

\begin{definition}
For a partition ${\bf a}$, we denote by ${\bf a}^{+}$ the partition obtained from ${\bf a}$ by removing all zero parts and rearranging the non-zero parts in weakly decreasing order.
\end{definition}

For objects $F$ and $F'$ of ${\rm Jl}_{\bf d}(V)$, if $F$ and $F'$ are isomorphic as a representation of quiver, then $F$ and $F'$ have the 
same $GL(V)$-orbits, which leads to the following proposition.

\begin{prop}\label{sec:3.6} 
${\bf d}$ is a dimension vector of finite type in $\mathscr{J}$ if and only if any summand ${\bf d'}$ of ${\bf d}$ on $\mathscr{J}$ 
satisfies $Q({\bf d'})\geq 1$.
\end{prop}
\begin{proof}
First, suppose {\bf d} is of finite type in $\mathscr{J}$.
Then any summand ${\bf d'}$ of {\bf d} on $\mathscr{J}$ is of finite type. 
Therefore, satisfy $Q({\bf d'})\geq 1$.
Conversely, any summand ${\bf d'}$ of {\bf d} on $\mathscr{J}$ satisfies $Q({\bf d'})\geq 1$. 
Any object $F$ of ${\mathscr{J}}$ has a unique indecomposable decomposition. 
In this case, from the assumptions and Proposition \ref{sec:2.1}, there are a finite number of isomorphism classes of indecomposable 
summands of $F$.
Therefore, {\bf d} is of finite type.
\end{proof}

For a dimension vector ${\bf d=(a, b, c)}\in {\Lambda}^{J}$, we denote by 
\[
{\bf d^{+}}:=({\bf a^{+}, ({b'}^{+}, |c'|), ({c'}^{+}, |b'|})).
\]
\begin{thm}\label{sec:3.8}
${\bf d=(a, b, c)}\in {\Lambda}^{J}$ satisfies either ${\bf d}$ is of finite type in $\mathscr{J}$ 
or ${\bf d}$ has at least one summand ${\bf d'}$ on $\mathscr{J}$ where ${\bf {d'}^{+}}$ is one of the following$:$\\
$((2^3), (1^3, 3), (1^3, 3))$, $((3^3), (2^2, 5), (1^5, 4))$, $((3^3), (1^5, 4), (2^2, 5))$, \\
$((1^4), (1^2, 2), (1^2, 2))$, $((2^4), (3, 5), (1^5, 3))$, $((2^4), (1^5, 3), (3, 5))$, \\
$((1^5), (2, 3), (1^3, 2))$, $((1^5), (1^3, 2), (2, 3))$, $((1^7), (3, 4), (2^2, 3))$, \\
$((1^7), (2^2, 3), (3, 4))$. 
\end{thm}
\begin{proof}
First, we will show that $0\neq {\bf d=(a, b, c)}\in {\Lambda}^{J}$ satisfies at least one of the following three conditions by dividing the cases for ${\bf d}$. \\ 
(I) $Q({\bf d})\geq 1$. \\ 
(I\hspace{-.1em}I) ${\bf d}$ has at least one summand ${\bf d'}$ on $\mathscr{J}$ where ${\bf {d'}^{+}}$ is one of \\ 
$((2^3), (1^3, 3), (1^3, 3))$, $((3^3), (2^2, 5), (1^5, 4))$, $((1^4), (1^2, 2), (1^2, 2))$, \\
$((2^4), (3, 5), (1^5, 3))$, $((1^5), (2, 3), (1^3, 2))$, or $((1^7), (3, 4), (2^2, 3))$. \\ 
(I\hspace{-.1em}I\hspace{-.1em}I) ${\bf d}$ is of finite type in $\mathscr{F}$. 

Here, ${\bf d}\in {\Lambda}^{J}$ satisfying (I\hspace{-.1em}I\hspace{-.1em}I) satisfies (I), 
but the reason for separating \\
(I\hspace{-.1em}I\hspace{-.1em}I) and (I) is to simplify  the proof. 
The condition that {\bf d} satisfies (I\hspace{-.1em}I\hspace{-.1em}I) is given by \cite{mwz1999}, 
and {\bf d} is of finite type in $\mathscr{J}$.


In the following, we first present a diagram showing the division into cases. Next, we investigate the relationship between each case and conditions (I), (I\hspace{-.1em}I), and (I\hspace{-.1em}I\hspace{-.1em}I). Finally, we prove the theorem by using the relations.

Let $r, p$, and $q$ denote the number of parts of ${\bf a^{+}, b^{+}}$, and ${\bf c^{+}}$, respectively. 
We also assume without loss of generality that $p\leq q$.

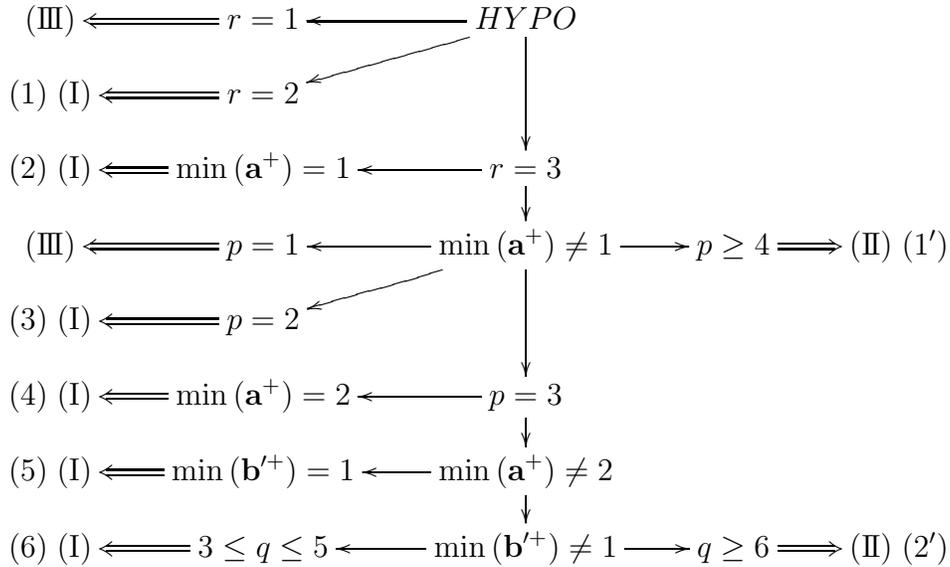
\[ 
\xymatrix@R=10pt{
 (\rm{I}\hspace{-.1em}\rm{I}\hspace{-.1em}\rm{I}) & r=1 \ar@{=>}[l] & HYPO \ar[l] \ar[dl] \ar[dd] \\
 (1)\; (\rm{I}) & r=2 \ar@{=>}[l] \\
 (2)\; (\rm{I}) & \min{({\bf a^{+}})}=1 \ar@{=>}[l] & r=3 \ar[l]  \ar[d]\\
 (\rm{I}\hspace{-.1em}\rm{I}\hspace{-.1em}\rm{I}) & p=1 \ar@{=>}[l] & \min{({\bf a^{+}})}\neq 1 \ar[dl] \ar[r] \ar[l] \ar[dd] & p\geq 4 \ar@{=>}[r] & (\rm{I}\hspace{-.1em}\rm{I})\; (1') \\ 
 (3)\; (\rm{I}) & p=2 \ar@{=>}[l] \\ 
 (4)\; (\rm{I}) & \min{({\bf a^{+}})}=2 \ar@{=>}[l] & p=3 \ar[d] \ar[l] \\ 
 (5)\; (\rm{I}) & \min{({\bf b'^{+}})}=1 \ar@{=>}[l] & \min{({\bf a^{+}})}\neq 2 \ar[l] \ar[d] \\ 
 (6)\; (\rm{I}) & 3\leq q\leq 5 \ar@{=>}[l] & \min{({\bf b'^{+}})}\neq 1 \ar[l] \ar[r] & q\geq 6 \ar@{=>}[r] & (\rm{I}\hspace{-.1em}\rm{I})\; (2')
} 
\] 

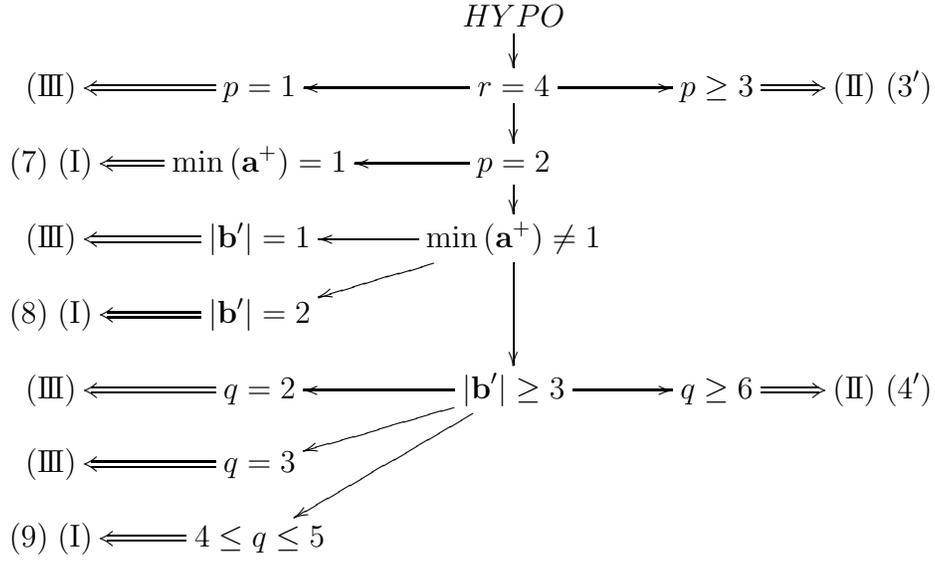
\[
\xymatrix@R=10pt{
 & & HYPO \ar[d] \\ 
 (\rm{I}\hspace{-.1em}\rm{I}\hspace{-.1em}\rm{I}) & p=1 \ar@{=>}[l] & r=4 \ar[l] \ar[r] \ar[d] & p\geq 3 \ar@{=>}[r] & (\rm{I}\hspace{-.1em}\rm{I})\; (3') \\
 (7)\; (\rm{I}) & \min{({\bf a^{+}})}=1 \ar@{=>}[l] & p=2 \ar[l] \ar[d] \\ 
 (\rm{I}\hspace{-.1em}\rm{I}\hspace{-.1em}\rm{I}) & |{\bf b'}|=1 \ar@{=>}[l] & \min{({\bf a^{+}})}\neq 1 \ar[l] \ar[dl] \ar[dd] \\ 
 (8)\; (\rm{I}) & |{\bf b'}|=2 \ar@{=>}[l] \\ 
 (\rm{I}\hspace{-.1em}\rm{I}\hspace{-.1em}\rm{I}) & q=2 \ar@{=>}[l] & |{\bf b'}|\geq 3 \ar[l] \ar[dl] \ar[ddl] \ar[r] & q\geq 6 \ar@{=>}[r] & (\rm{I}\hspace{-.1em}\rm{I})\; (4') \\ 
 (\rm{I}\hspace{-.1em}\rm{I}\hspace{-.1em}\rm{I}) & q=3 \ar@{=>}[l] \\ 
 (9)\; (\rm{I}) & 4\leq q\leq 5 \ar@{=>}[l] 
}
\]

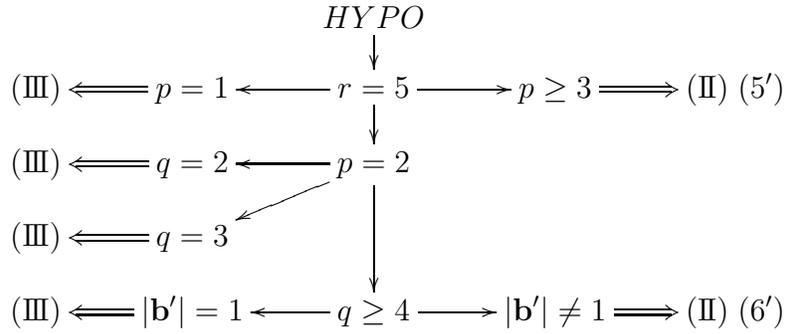
\[
\xymatrix@R=10pt{
 & & HYPO \ar[d] \\ 
 (\rm{I}\hspace{-.1em}\rm{I}\hspace{-.1em}\rm{I}) & p=1 \ar@{=>}[l] & r=5 \ar[l] \ar[r] \ar[d] & p\geq 3 \ar@{=>}[r] & (\rm{I}\hspace{-.1em}\rm{I})\; (5') \\ 
 (\rm{I}\hspace{-.1em}\rm{I}\hspace{-.1em}\rm{I}) & q=2 \ar@{=>}[l] & p=2 \ar[l] \ar[dl] \ar[dd] \\ 
 (\rm{I}\hspace{-.1em}\rm{I}\hspace{-.1em}\rm{I}) & q=3 \ar@{=>}[l] \\ 
 (\rm{I}\hspace{-.1em}\rm{I}\hspace{-.1em}\rm{I}) & |{\bf b'}|=1 \ar@{=>}[l] & q\geq 4 \ar[l] \ar[r] & |{\bf b'}|\neq 1 \ar@{=>}[r] & (\rm{I}\hspace{-.1em}\rm{I})\; (6')
}
\]

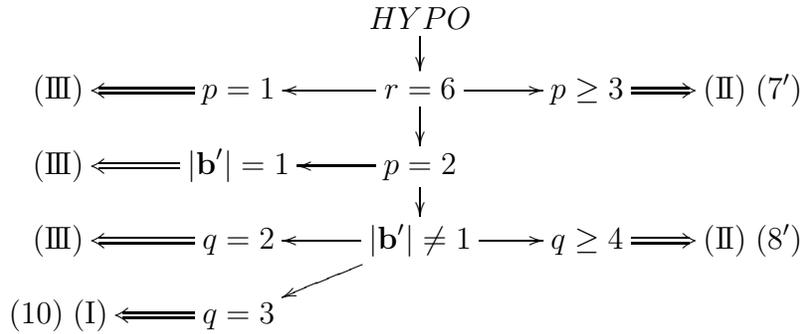
\[
\xymatrix@R=10pt{
 & & HYPO \ar[d] \\ 
 (\rm{I}\hspace{-.1em}\rm{I}\hspace{-.1em}\rm{I}) & p=1 \ar@{=>}[l] & r=6 \ar[l] \ar[r] \ar[d] & p\geq 3 \ar@{=>}[r] & (\rm{I}\hspace{-.1em}\rm{I})\; (7') \\ 
 (\rm{I}\hspace{-.1em}\rm{I}\hspace{-.1em}\rm{I}) & |{\bf b'}|=1 \ar@{=>}[l] & p=2 \ar[l] \ar[d] \\ 
 (\rm{I}\hspace{-.1em}\rm{I}\hspace{-.1em}\rm{I}) & q=2 \ar@{=>}[l] & |{\bf b'}|\neq 1 \ar[l] \ar[dl] \ar[r] & q\geq 4 \ar@{=>}[r] & (\rm{I}\hspace{-.1em}\rm{I})\; (8') \\ 
 (10)\; (\rm{I}) & q=3 \ar@{=>}[l] 
}
\]

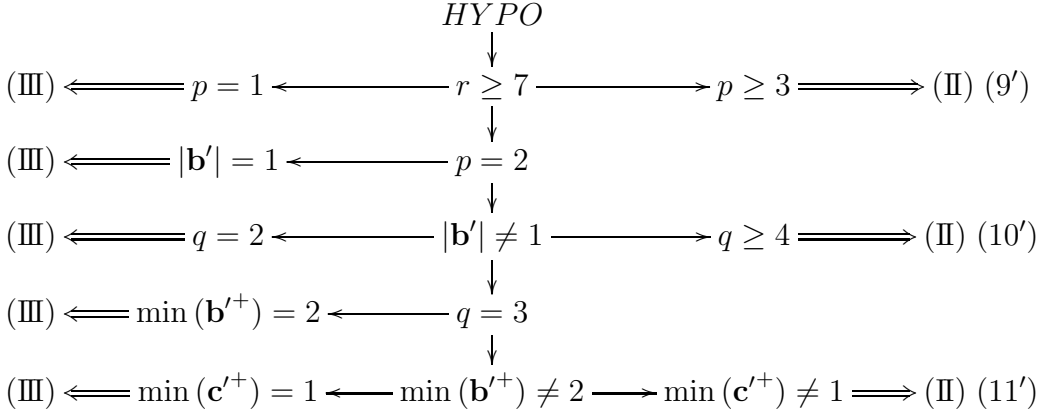
\[
\xymatrix@R=10pt{
 & & HYPO \ar[d] \\ 
 (\rm{I}\hspace{-.1em}\rm{I}\hspace{-.1em}\rm{I}) & p=1 \ar@{=>}[l] & r\geq 7 \ar[l] \ar[r] \ar[d] & p\geq 3 \ar@{=>}[r] & (\rm{I}\hspace{-.1em}\rm{I})\; (9') \\ 
 (\rm{I}\hspace{-.1em}\rm{I}\hspace{-.1em}\rm{I}) & |{\bf b'}|=1  \ar@{=>}[l] & p=2 \ar[l] \ar[d] \\ 
 (\rm{I}\hspace{-.1em}\rm{I}\hspace{-.1em}\rm{I}) & q=2 \ar@{=>}[l] & |{\bf b'}|\neq 1 \ar[l] \ar[r] \ar[d] & q\geq 4 \ar@{=>}[r] & (\rm{I}\hspace{-.1em}\rm{I})\; (10') \\ 
 (\rm{I}\hspace{-.1em}\rm{I}\hspace{-.1em}\rm{I}) & \min{({\bf {b'}^{+}})}=2 \ar@{=>}[l] & q=3 \ar[l] \ar[d] \\ 
 (\rm{I}\hspace{-.1em}\rm{I}\hspace{-.1em}\rm{I}) & \min{({\bf {c'}^{+}})}=1 \ar@{=>}[l] & \min{({\bf {b'}^{+}})}\neq 2 \ar[l] \ar[r] & \min{({\bf {c'}^{+}})}\neq 1 \ar@{=>}[r] & (\rm{I}\hspace{-.1em}\rm{I})\; (11')
}
\]

We first show that ${\bf d=(a, b, c)}$ satisfies (I) in cases (1) to (10). 
For $(A, B, C)\in {\rm Jl}_{\bf (a, b, c)}(V)$, let the dimensions of $B_{\ell (b)-1}$ and $C_{\ell (c)-1}$ be $x:=\dim{B_{\ell (b)-1}}$ and $y:=\dim{C_{\ell (c)-1}}$. 


Then 
\begin{align*}
{\bf d=(a, b, c)}=({\bf a}, ({\bf b'}:=(b_1, \cdots , b_{\ell ({\bf b})-1}), y), ({\bf c'}:=(c_1, \cdots , c_{\ell ({\bf c})-1}), x)). 
\end{align*}
Therefore, we get 
\begin{align*}
Q({\bf d})&=(||{\bf a}||^2+||{\bf b}||^2+||{\bf c}||^2-(x+y)^2)/2\\
&=(||{\bf a}||^2+||{\bf b'}||^2+||{\bf c'}||^2-2xy)/2.
\end{align*}

Where ${\bf a}$, ${\bf b'}$, and ${\bf c'}$ are integer partitions of $x+y$, $x$, and $y$, respectively. 
Now consider ${\bf a}^{\bbR}$, ${\bf b'}^{\bbR}$, and ${\bf c'}^{\bbR}$ as partitions of $x+y$, $x$, and $y$ in the range of non-negative real numbers, respectively. 
Here, it assumes that the conditions (1) to (10) (for the number of non-zero parts and the minimum value of non-zero parts) are satisfied for each case. 
Then there exists a ${\bf d^{\bbR}:=(a^{\bbR}, {b'}^{\bbR}, {c'}^{\bbR})}$ such that $Q({\bf d}^{\bbR}):=(||{\bf a^{\bbR}}||^2+||{\bf {b'}}^{\bbR}||^2+||{\bf {c'}}^{\bbR}||^2-2xy)/2$ is minimized. 
Also, at this time, ${\bf a}^{\bbR}$, ${\bf b'}^{\bbR}$, and ${\bf c'}^{\bbR}$ are unique, except for the reordering of their respective parts. 
Here, since $Q({\bf d})\geq 1$ if $Q({\bf d^{\bbR}})>0$, we show that $Q({\bf d^{\bbR}})>0$ for ${\bf d}$ in the case (1) to (10). 
The table of ${\bf d^{{\bbR}^{+}}}$ and $Q({\bf d^{\bbR}})$ for the case (1) to (10) is given bellow. 

\begin{align*}
&(1)\; {\bf d^{{\bbR}^{+}}}\! \! \! \! =((((x+y)/2)^2), (1^x, y), (1^y, x)),\\
&\; \; \; \; \; ( x+y\geq 2,\; y\geq x), \\
&\; \; \; \; \; 2Q({\bf d^{{\bbR}^{+}}})=-2xy+((x+y)^2)/2+x+y \\
&\; \; \; \; \; \; \; \; \; \; \; \; \; \; \; \; \; \; =((x-y+1)^2)/2+(4y-1)/2. \\
&(2)\; {\bf d^{{\bbR}^{+}}}\! \! \! \! =((1, ((x+y-1)/2)^2), (1^x, y), (1^y, x)), \\
&\; \; \; \; \; (x+y\geq 3,\; y\geq x),\\
&\; \; \; \; \; 2Q({\bf d^{{\bbR}^{+}}})=-2xy+((x+y-1)^2)/2+x+y+1=((x-y)^2)/2+3/2. \\ 
&(3)\; {\bf d^{{\bbR}^{+}}}\! \! \! \! =((((x+y)/3)^3), (x, y), (1^y, x)), \\
&\; \; \; \; \; (x+y\geq 6,\; x\geq 1,\; y\geq 1),\\
&\; \; \; \; \; 2Q({\bf d^{{\bbR}^{+}}})=-2xy+((x+y)^2)/3+x^2+y=((2x-y)^2)/3+y. \\ 
&(4)\; {\bf d^{{\bbR}^{+}}}\! \! \! \! =((2, ((x+y-2)/2)^2), ((x/2)^2, y), (1^y, x)), \\
&\; \; \; \; \; (x+y\geq 6,\; x\geq 2,\; y\geq 2),\\
&\; \; \; \; \; 2Q({\bf d^{{\bbR}^{+}}})=-2xy+((x+y-2)^2)/2+x^2/2+y+4 \\
&\; \; \; \; \; \; \; \; \; \; \; \; \; \; \; \; \; \; =((2x-y-2)^2)/4+(y^2-8y+20)/4. \\ 
&(5)\; {\bf d^{{\bbR}^{+}}}\! \! \! \! =((((x+y)/3)^3), (1, x-1 , y), (1^y, x)), \\
&\; \; \; \; \; (x+y\geq 9,\; x\geq2,\; y\geq 2), \\
&\; \; \; \; \; 2Q({\bf d^{{\bbR}^{+}}})=-2xy+((x+y)^2)/3+(x-1)^2+y+1\\
&\; \; \; \; \; \; \; \; \; \; \; \; \; \; \; \; \; \; =((4x-2y-3)^2)/12+5/4. \\ 
&(6)\; (q=3)\; {\bf d^{{\bbR}^{+}}}\! \! \! \! =((((x+y)/3)^3), ((x/2)^2, y), ((y/2)^2, x)), \\
&\; \; \; \; \; \; \; \; \; \; \; \; \; \; \; \; \; (x+y\geq 9,\; x\geq 4,\; y\geq 2), \\
&\; \; \; \; \; \; \; \; \; \; \; \; \; \; \; \; \; 2Q({\bf d^{{\bbR}^{+}}})= -2xy+((x+y)^2)/3+x^2/2+y^2/2\\
&\; \; \; \; \; \; \; \; \; \; \; \; \; \; \; \; \; \; \; \; \; \; \; \; \; \; \; \; \; \; =((5x-4y)^2)/30+(3y^2)/10. \\
&\; \; \; \; \; (q=4)\; {\bf d^{{\bbR}^{+}}}\! \! \! \! =((((x+y)/3)^3), ((x/2)^2, y), ((y/3)^3, x)), \\
&\; \; \; \; \; \; \; \; \; \; \; \; \; \; \; \; \; (x+y\geq 9,\; x\geq 4,\; y\geq 3), \\
&\; \; \; \; \; \; \; \; \; \; \; \; \; \; \; \; \; 2Q({\bf d^{{\bbR}^{+}}})=-2xy+((x+y)^2)/3+x^2/2+y^2/3\\
&\; \; \; \; \; \; \; \; \; \; \; \; \; \; \; \; \; \; \; \; \; \; \; \; \; \; \; \; \; \; =((5x-4y)^2)/30+(2y^2)/15. \\
\end{align*}
\begin{align*}
&\; \; \; \; \; (q=5)\; {\bf d^{{\bbR}^{+}}}\! \! \! \! =((((x+y)/3)^3), ((x/2)^2, y), ((y/4)^4, x)),\\ 
&\; \; \; \; \; \; \; \; \; \; \; \; \; \; \; \; \; (x+y\geq 9,\; x\geq 4,\; y\geq 4), \\
&\; \; \; \; \; \; \; \; \; \; \; \; \; \; \; \; \; 2Q({\bf d^{{\bbR}^{+}}})=-2xy+((x+y)^2)/3+x^2/2+y^2/4\\
&\; \; \; \; \; \; \; \; \; \; \; \; \; \; \; \; \; \; \; \; \; \; \; \; \; \; \; \; \; \; =((5x-4y)^2)/30+(y^2)/20. \\
&(7)\; {\bf d^{{\bbR}^{+}}}\! \! \! \! =((1, ((x+y-1)/3)^3), (x, y), (1^y, x)),\\
&\; \; \; \; \; (x+y\geq 4,\; x\geq 1,\; y\geq 1), \\ 
&\; \; \; \; \; 2Q({\bf d^{{\bbR}^{+}}})= -2xy+((x+y-1)^2)/3+x^2+y+1\\
&\; \; \; \; \; \; \; \; \; \; \; \; \; \; \; \; \; \; =((4x-2y-1)^2)/12+5/4. \\ 
&(8)\; {\bf d^{{\bbR}^{+}}}\! \! \! \! =((((2+y)/4)^4), (2, y), (1^y, 2)), \\ 
&\; \; \; \; \; (y\geq 6), \\
&\; \; \; \; \; 2Q({\bf d^{{\bbR}^{+}}})= -4y+((2+y)^2)/4+2^2+y=(y^2-8y+20)/4. \\ 
&(9)\; (q=4)\; {\bf d^{{\bbR}^{+}}}\! \! \! \! =((((x+y)/4)^4), (x, y), ((y/3)^3, x)), \\ 
&\; \; \; \; \; \; \; \; \; \; \; \; \; \; \; \; \; (x+y\geq 8,\; x\geq 3,\; y\geq 3),\\
&\; \; \; \; \; \; \; \; \; \; \; \; \; \; \; \; \; 2Q({\bf d^{{\bbR}^{+}}})= -2xy+((x+y)^2)/4+x^2+y^2/3\\
&\; \; \; \; \; \; \; \; \; \; \; \; \; \; \; \; \; \; \; \; \; \; \; \; \; \; \; \; \; \; =((5x-3y)^2)/20+(2y^2)/15.\\
&\; \; \; \; \; (q=5)\; {\bf d^{{\bbR}^{+}}}\! \! \! \! =((((x+y)/4)^4), (x, y), ((y/4)^4, x)),\\
&\; \; \; \; \; \; \; \; \; \; \; \; \; \; \; \; \; (x+y\geq 8,\; x\geq 3,\; y\geq 4),\\
&\; \; \; \; \; \; \; \; \; \; \; \; \; \; \; \; \; 2Q({\bf d^{{\bbR}^{+}}})=-2xy+((x+y)^2)/4+x^2+y^2/4\\
&\; \; \; \; \; \; \; \; \; \; \; \; \; \; \; \; \; \; \; \; \; \; \; \; \; \; \; \; \; \; =((5x-3y)^2)/20+(y^2)/20. \\
&(10)\; {\bf d^{{\bbR}^{+}}}\! \! \! \! =((((x+y)/6)^6), (x, y), ((y/2)^2, x)),\\
&\; \; \; \; \; (x+y\geq 6,\; x\geq 2,\; y\geq 2),\\
&\; \; \; \; \; 2Q({\bf d^{{\bbR}^{+}}})= -2xy+((x+y)^2)/6+x^2+y^2/2\\
&\; \; \; \; \; \; \; \; \; \; \; \; \; \; \; \; \; \; =((7x-5y)^2)/42+(y^2)/14. 
\end{align*}

The reason why it is $Q({\bf d^{\bbR}})>0$ follows immediately from the fact that solving for $Q({\bf d^{\bbR}})\leq 0$ yields $(x, y)=(0, 0)$ or no solution.

Next, for ${\bf d}$ in the case $(1')$ to $(11')$, we show that there is a summand ${\bf d'}$ of ${\bf d}$ on $\mathscr{J}$ such that ${\bf d'^{+}}$ is one of the following 
$((2^3), (1^3, 3), (1^3, 3))$, $((3^3), (2^2, 5), (1^5, 4))$, $((1^4), (1^2, 2), (1^2, 2))$, $((2^4), (3, 5), (1^5, 3))$, \\
$((1^5), (2, 3), (1^3, 2))$, or $((1^7), (3, 4), (2^2, 3))$.\\

In the case of $(1')$, ${\bf d=(a, b, c)}\in {\Lambda}^{J}$ contained in the case classification of $(1')$ satisfies: 
$q\geq p\geq 4$, $r=3$, $min({\bf a^{+}})\neq 1$, so $b_{\ell ({\bf b})}, c_{\ell ({\bf c})}\geq 3$, and $\exists  a_i, a_j, a_k\geq 2$ for some $1\leq i<j<k\leq \ell ({\bf a})$, 
$\exists b_{i'}, b_{j'},b_{k'}\geq 1$ for some $1\leq i'<j'<k'<\ell ({\bf b})$, 
$\exists c_{i''}, c_{j''}, c_{k''}\geq 1$ for some $1\leq i''<j''<k''<\ell ({\bf c})$. 
Therefore,  ${\bf d}$ has a summand ${\bf d'}=((\cdots , {\overset{i}{2}}, \cdots , \overset{j}{2}, \cdots , \overset{k}{2}, \cdots)$, 
$(\cdots , \overset{i'}{1}, \cdots , \overset{j'}{1}, \cdots , \overset{k'}{1}, \cdots , 3)$, 
$(\cdots , \overset{i''}{1}, \cdots , \overset{j''}{1}, \cdots , \overset{k''}{1}, \cdots , 3))$ on $\mathscr{J}$, \\
where ${\bf {d'}^{+}}=((2^3), (1^3, 3), (1^3, 3))$. \\

In the case of $(2')$, ${\bf d=(a, b, c)}\in {\Lambda}^{J}$ contained in the case classification of $(2')$ satisfies: 
$p=3$, $q\geq 6$, $r=3$, $min({\bf a^{+}})\neq 1, 2$, $min({\bf b'^{+}})\neq 1$, 
so $b_{\ell ({\bf b})}\geq 5$, $c_{\ell ({\bf c})}\geq 4$, and $\exists  a_i, a_j, a_k\geq 3$ for some $1\leq i<j<k\leq \ell ({\bf a})$, 
$\exists b_{i'}, b_{j'}\geq 2$ for some $1\leq i'<j'<\ell ({\bf b})$, 
$\exists c_{i''}, c_{j''}, c_{k''}, c_{l''}, c_{m''}\geq 1$ for some $1\leq i''<j''<k''<l''<m''<\ell ({\bf c})$. 
Therefore,  ${\bf d}$ has a summand ${\bf d'}=((\cdots , {\overset{i}{3}}, \cdots , \overset{j}{3}, \cdots , \overset{k}{3}, \cdots)$, 
$(\cdots , \overset{i'}{2}, \cdots , \overset{j'}{2}, \cdots , 5)$, \\
$(\cdots , \overset{i''}{1}, \cdots , \overset{j''}{1}, \cdots , \overset{k''}{1}, \cdots , \overset{l''}{1}, \cdots , \overset{m''}{1}, \cdots , 4))$ on $\mathscr{J}$, \\
where ${\bf {d'}^{+}}=((3^3), (2^2, 5), (1^5, 4))$. \\ 

In the case of $(3')$, $(5')$, $(7')$, $(9')$, ${\bf d=(a, b, c)}\in {\Lambda}^{J}$ included in a case classification that satisfies at least one of $(3')$, $(5')$, $(7')$, or $(9')$ satisfies: 
$q\geq p\geq 3$, $r\geq 4$, 
so $b_{\ell ({\bf b})}, c_{\ell ({\bf c})}\geq 2$, and $\exists  a_i, a_j, a_k, a_l\geq 1$ for some $1\leq i<j<k<l\leq \ell ({\bf a})$, 
$\exists b_{i'}, b_{j'}\geq 1$ for some $1\leq i'<j'<\ell ({\bf b})$, 
$\exists c_{i''}, c_{j''}\geq 1$ for some $1\leq i''<j''<\ell ({\bf c})$. 
Therefore,  ${\bf d}$ has a summand ${\bf d'}=((\cdots , {\overset{i}{1}}, \cdots , \overset{j}{1}, \cdots , \overset{k}{1}, \cdots , \overset{l}{1}, \cdots)$, 
$(\cdots , \overset{i'}{1}, \cdots , \overset{j'}{1}, \cdots , 2)$, \\
$(\cdots , \overset{i''}{1}, \cdots , \overset{j''}{1}, \cdots , 2))$ on $\mathscr{J}$, 
where ${\bf {d'}^{+}}=((1^4), (1^2, 2), (1^2, 2))$. \\ 

In the case of $(4')$, ${\bf d=(a, b, c)}\in {\Lambda}^{J}$ contained in the case classification of $(4')$ satisfies: 
$p=2$, $q\geq 6$, $r=4$, $min({\bf a^{+}})\neq 1$, $|{\bf b'}|\geq 3$, 
so $b_{\ell ({\bf b})}\geq 5$, $c_{\ell ({\bf c})}=|{\bf b'}|\geq 3$, and 
$\exists  a_i, a_j, a_k, a_l\geq 2$ for some $1\leq i<j<k<l\leq \ell ({\bf a})$, 
$\exists b_{i'}\geq 3$ for some $1\leq i'<\ell ({\bf b})$, 
$\exists c_{i''}, c_{j''}, c_{k''}, c_{l''}, c_{m''}\geq 1$ for some $1\leq i''<j''<k''<l''<m''<\ell ({\bf c})$. 
Therefore,  ${\bf d}$ has a summand ${\bf d'}=((\cdots , {\overset{i}{2}}, \cdots , \overset{j}{2}, \cdots , \overset{k}{2}, \cdots, \overset{l}{2}, \cdots)$, 
$(\cdots , \overset{i'}{3}, \cdots , 5)$, \\
$(\cdots , \overset{i''}{1}, \cdots , \overset{j''}{1}, \cdots , \overset{k''}{1}, \cdots , \overset{l''}{1}, \cdots , \overset{m''}{1}, \cdots , 3))$ on $\mathscr{J}$, \\
where ${\bf {d'}^{+}}=((2^4), (3, 5), (1^5, 3))$. \\ 

In the case of $(6')$, $(8')$, $(10')$, ${\bf d=(a, b, c)}\in {\Lambda}^{J}$ included in a case classification that satisfies at least one of $(6')$, $(8')$, or $(10')$ satisfies: 
$p=2$, $q\geq 4$, $r\geq 5$, $|{\bf b'}|\neq 1$, 
so $b_{\ell ({\bf b})}\geq 3$, $c_{\ell ({\bf c})}\geq 2$, and $\exists  a_i, a_j, a_k, a_l, a_m\geq 1$ for some $1\leq i<j<k<l<m\leq \ell ({\bf a})$, 
$\exists b_{i'}\geq 2$ for some $1\leq i'<\ell ({\bf b})$, 
$\exists c_{i''}, c_{j''}, c_{k''}\geq 1$ for some $1\leq i''<j''<k''<\ell ({\bf c})$. 
Therefore,  ${\bf d}$ has a summand ${\bf d'}=((\cdots , {\overset{i}{1}}, \cdots , \overset{j}{1}, \cdots , \overset{k}{1}, \cdots , \overset{l}{1}, \cdots , \overset{m}{1}, \cdots)$, 
$(\cdots , \overset{i'}{2}, \cdots , 3)$, 
$(\cdots , \overset{i''}{1}, \cdots , \overset{j''}{1}, \cdots , \overset{k''}{1}, \cdots , 2))$ on $\mathscr{J}$, 
where ${\bf {d'}^{+}}=((1^5), (2, 3), (1^3, 2))$. \\ 

In the case of $(11')$, ${\bf d=(a, b, c)}\in {\Lambda}^{J}$ contained in the case classification of $(11')$ satisfies: 
$p=2$, $q=3$, $r\geq 7$, $min({\bf b'^{+}})\neq 1, 2$, $min({\bf c'^{+}})\neq 1$, 
so $b_{\ell ({\bf b})}\geq 4$, $c_{\ell ({\bf c})}\geq 3$, and 
$\exists  a_i, a_j, a_k, a_l, a_m, a_n, a_o\geq 1$ for some $1\leq i<j<k<l<m<n<o\leq \ell ({\bf a})$, 
$\exists b_{i'}\geq 3$ for some $1\leq i'<\ell ({\bf b})$, 
$\exists c_{i''}, c_{j''}\geq 2$ for some $1\leq i''<j''<\ell ({\bf c})$. 
Therefore,  ${\bf d}$ has a summand ${\bf d'}=((\cdots , {\overset{i}{1}}, \cdots , \overset{j}{1}, \cdots , \overset{k}{1}, \cdots , {\overset{l}{1}}, \cdots , \overset{m}{1}, \cdots , \overset{n}{1}, \cdots , {\overset{o}{1}}, \cdots)$, 
$(\cdots , \overset{i'}{3}, \cdots , 4)$, 
$(\cdots , \overset{i''}{2}, \cdots , \overset{j''}{2}, \cdots , 3))$ on $\mathscr{J}$, where ${\bf {d'}^{+}}=((1^7), (3, 4), (2^2, 3))$. \\ 

Now let ${\bf d}\in {\Lambda}^{J}$ be such that ${\bf d^{+}}$ is one of the following: \\
$((2^3), (1^3, 3), (1^3, 3))$, $((3^3), (2^2, 5), (1^5, 4))$, $((1^4), (1^2, 2), (1^2, 2))$, \\
$((2^4), (3, 5), (1^5, 3))$, $((1^5), (2, 3), (1^3, 2))$, and $((1^7), (3, 4), (2^2, 3))$. 
\\Then $Q({\bf d})=0$, so ${\bf d}\in {\Lambda}^{J}$ in case $(\rm{I}\hspace{-.1em}\rm{I})$ is of infinite-type in $\mathscr{J}$ from Proposition \ref{sec:3.6}.

Next, we show that any ${\bf d}\in {\Lambda}^{J}$ in cases (1) to (10) does not satisfy $(\rm{I}\hspace{-.1em}\rm{I})$. Then $Q({\bf d'})\geq 1$ for any summand ${\bf d'}$ of {\bf d} on $\mathscr{J}$ from the division into cases. Therefore, {\bf d} in the case (1) to (10) is of finite type in $\mathscr{J}$ from Proposition \ref{sec:3.6}.

In the case of (1), ${\bf d=(a, b, c)}\in {\Lambda}^{J}$ contained in the case classification of (1) satisfies $r=2$. Therefore, {\bf d} does not satisfy $(\rm{I}\hspace{-.1em}\rm{I})$.

In the case of (2), ${\bf d=(a, b, c)}\in {\Lambda}^{J}$ contained in the case classification of (2) satisfies $r=3$, and $\min{({\bf a^{+}})}=1$. Therefore, {\bf d} does not satisfy $(\rm{I}\hspace{-.1em}\rm{I})$.

In the case of (3), ${\bf d=(a, b, c)}\in {\Lambda}^{J}$ contained in the case classification of (3) satisfies $r=3$, and $p=2$. Therefore, {\bf d} does not satisfy $(\rm{I}\hspace{-.1em}\rm{I})$.

In the case of (4), ${\bf d=(a, b, c)}\in {\Lambda}^{J}$ contained in the case classification of (4) satisfies $r=3$, $p=3$ and $\min{({\bf a^{+}})}=2$. Therefore, {\bf d} does not satisfy $(\rm{I}\hspace{-.1em}\rm{I})$.

In the case of (5), ${\bf d=(a, b, c)}\in {\Lambda}^{J}$ contained in the case classification of (5) satisfies $r=3$, $p=3$ and $\min{({\bf {b'}^{+}})}=1$. Therefore, {\bf d} does not satisfy $(\rm{I}\hspace{-.1em}\rm{I})$.

In the case of (6), ${\bf d=(a, b, c)}\in {\Lambda}^{J}$ contained in the case classification of (6) satisfies $r=3$, $p=3$ and $3\leq q\leq 5$. Therefore, {\bf d} does not satisfy $(\rm{I}\hspace{-.1em}\rm{I})$.

In the case of (7), ${\bf d=(a, b, c)}\in {\Lambda}^{J}$ contained in the case classification of (7) satisfies $r=4$, $p=2$ and $\min{({\bf a^{+}})}=1$. Therefore, {\bf d} does not satisfy $(\rm{I}\hspace{-.1em}\rm{I})$.

In the case of (8), ${\bf d=(a, b, c)}\in {\Lambda}^{J}$ contained in the case classification of (8) satisfies $r=4$, $p=2$ and $|{\bf b'}|=2$. Therefore, {\bf d} does not satisfy $(\rm{I}\hspace{-.1em}\rm{I})$.

In the case of (9), ${\bf d=(a, b, c)}\in {\Lambda}^{J}$ contained in the case classification of (9) satisfies $r=4$, $p=2$ and $4\leq q\leq 5$. Therefore, {\bf d} does not satisfy $(\rm{I}\hspace{-.1em}\rm{I})$.

In the case of (10), ${\bf d=(a, b, c)}\in {\Lambda}^{J}$ contained in the case classification of (10) satisfies $r=6$, $p=2$ and $q=3$. Therefore, {\bf d} does not satisfy $(\rm{I}\hspace{-.1em}\rm{I})$.

\end{proof}

\begin{lem}\label{sec:3.9} 
Fix the natural numbers $m$ and $n$ such that $m\leq n$.  
Let ${\bf a}$ be an integer partition of n such that $\ell ({\bf a^{+}})=m$. In this case, the following are equivalent. \\ 
$(1)$ ${\bf a}$ is an integer partition of $n$ such that $||{\bf a}||^{2}$ is minimum.\\
$(2)$ The difference between any two parts of ${\bf a^{+}}$ is less than or equal to $1$. 
\end{lem}
\begin{proof}
First suppose ${\bf a}$ is an integer partition of $n$ such that $||{\bf a}||^2$ is minimum. 
Assume now that there exists non-zero parts $a_1, a_2$ of ${\bf a}$ such that $a_1-a_2\geq 2$. Then 
\begin{align*}
(a_1-1)^2+(a_2+1)^2&={a_1}^2+{a_2}^2-2a_1+2a_2+2\\
&={a_1}^2+{a_2}^2-2(a_1-a_2-1)<{a_1}^2+{a_2}^2.
\end{align*}
Here, let ${\bf b}$ be the integer partition of $n$ that can be made by replacing $a_1$ with $a_1-1$ and $a_2$ with $a_2+1$ among the parts of ${\bf a}$. Then $||{\bf b}||^2<||{\bf a}||^2$, which contradicts the assumption.

Conversely, if the difference of any two parts of ${\bf a}$ is less than or equal to one, 
then there is a unique integer partition of $n$ such that the difference of any two non-zero parts is less than or equal to 1, except for reordering. 
Also, there exists an integer partition of $n$ that minimizes the sum of the squares of each part. 
Therefore, from the proof that $(2)$ is true if $(1)$ is shown earlier, we can see that $||{\bf a}||^2$ is minimum.
\end{proof}

\begin{thm}
${\bf d=(a, b, c)}\in {\Lambda}^{J}$ is of finite type in $\mathscr{J}$ and satisfies $Q({\bf d})=1$ if and only if the $({\bf b}$ and {\bf c} unordered$)$ ${\bf d^{+}}$ is one of the following$:$ 
\begin{align*}
&((1), (1), (1, 0)), \; \; ((1^6), (2, 4), (2^2, 2)),\\
&((1^{2x+1}), (x, x+1), (x, 1, x))\; x\geq 2, \\
&((1^{2x}), (x, x), (x-1, 1, x))\; x\geq 2, \\
&((1^x), (1, x-1), (1^{x-1}, 1))\; x\geq 2, \\
&((x, x-1, 1), (1^{x}, x), (1^{x}, x))\; x\geq 2,\\
&((x^2, 1), (1^{x}, x+1), (1^{x+1}, x))\; x\geq 2, \\
&((2^3), (1^2, 4), (1^4, 2)),\; \; ((2^3), (2, 1, 3), (1^3, 3)),\; \; ((3, 2^2), (2, 1, 4), (1^4, 3)),\\
&((3^2, 2), (2, 1, 5), (1^5, 3)),\; \; ((3^2, 2), (2^2, 4), (1^4, 4)),\; \; ((4, 3, 2), (2^2, 5), (1^5, 4)),\\
&((4^2, 2), (2^2, 6), (1^6, 4)), \\
&((x^3), (x-1, 1, 2x), (1^{2x}, x))\; x\geq 3, \\
&((x^2, x-1), (x-1, 1, 2x-1), (1^{2x-1}, x))\; x\geq 4, \\
&((x, (x-1)^2), (x-1, 1, 2x-2), (1^{2x-2}, x))\; x\geq 4, \\ 
&(((x-1)^3), (x-1, 1, 2x-3), (1^{2x-3}, x))\; x\geq 4, \\ 
&((3^3), (2^2, 5), (2, 1^3, 4)), \\
&(((x-1)^3, 1), (x, 2x-2), (1^{2x-2}, x))\; x\geq 2,\\
&((x, (x-1)^2, 1), (x, 2x-1), (1^{2x-1}, x))\; x\geq 2,\\
&((x^2, x-1, 1), (x, 2x), (1^{2x}, x))\; x\geq 2,\\
&((x^3, 1), (x, 2x+1), (1^{2x+1}, x))\; x\geq 2,\\
&((2^4), (2, 6), (1^6, 2)),\; \; ((2^4), (3, 5), (2, 1^3, 3)),\; \; ((2, 1^5), (3, 4), (2^2, 3)).\\
\end{align*}
\end{thm}
\begin{proof}
First suppose ${\bf d}\in {\Lambda}^{J}$ is of finite type in $\mathscr{J}$ and satisfies $Q({\bf d})=1$. 
If {\bf d} is of finite type in $\mathscr{F}$ and satisfies $Q({\bf d})=1$, then ${\bf d^{+}}$ is one of the following from \cite{mwz1999}:
\begin{align*} 
&((1), (1), (1, 0)), \; \; ((1^6), (2, 4), (2^2, 2)), \\
&((1^{2x+1}), (x, x+1), (x, 1, x))\; x\geq 2, \\
&((1^{2x}), (x, x), (x-1, 1, x))\; x\geq 2, \\
&((1^x), (1, x-1), (1^{x-1}, 1))\; x\geq 2. 
\end{align*}
Therefore, it is sufficient to find all the {\bf d} in the case (1) to (10) of Theorem \ref{sec:3.8}  that satisfy $Q({\bf d})=1$. \\
Then, for ${\bf d^{{\bbR}^{+}}}$ in each case from (1) to (10), 
we find the pair $(x, y)$ of natural numbers that satisfies $0<Q({\bf d^{{\bbR}^{+}}})\leq 1$. 
For this $(x, y)$, is there ${\bf d}\in {\Lambda}^{J}$ that satisfies the conditions of each case classification? 
If so, use Lemma \ref{sec:3.9}  to find the ${\bf d}\in {\Lambda}^{J}$ for which $Q({\bf d})$ is the minimum. 
This ${\bf d}$ is unique except for the rearrangement of parts, and we only need to check whether it is $Q({\bf d})=1$ or not. \\
The table of $(x, y)$ satisfying $0<Q({\bf d^{{\bbR}^{+}}})\leq 1$ in each case from (1) to (10) is given bellow.
\begin{align*}
&(1)\; \; (x, y)=(1, 1), \\
&(2)\; \; (x, y)=(1, 1),\; (1, 2),\; (x, x-z)\; 1\geq z\geq -1,\; x\geq 2, \\
&(3)\; \; (x, y)=(1, 1),\; (1, 2), \\ 
&(4)\; \; (x, y)=(2, 2),\; (2, 3),\; (2, 4),\; (3, 3),\; (3, 4),\; (3, 5),\; (4, 4),\; (4, 5),\; (4, 6),  \\
&(5)\; \; (x, y)=(1, 1),\; (1, 2),\; (x, 2x-z)\; 3\geq z\geq 0,\; x\geq 2, \\
&(6)\; \; (q=3)\; \; (x, y)=(1, 1),\; (1, 2),\; (2, 1),\; (2, 2), \\
&\; \; \; \; \; \; \, (q=4)\; \; (x, y)=(1, 1),\; (1, 2),\; (2, 1),\; (2, 2),\; (2, 3),\; (3, 3), \\
&\; \; \; \; \; \; \, (q=5)\; \; (x, y)=(1, 1),\; (1, 2),\; (2, 1),\; (2, 2),\; (2, 3),\; (2, 4),\; (3, 2),\; (3, 3), \\ 
&\; \; \; \; \; \; \; \; \; \; \; \; \; \; \; \; \; \; \; \; \; \; \; \; \; \; \; \; \; \; \; \; \, (3, 4),\; (4, 4),\; (4, 5),\; (5, 6), \\
&(7)\; \; (x, y)=(1, 1),\; (1, 2),\; (1, 3),\; (x, 2x-z)\; 2\geq z\geq -1,\; x\geq 2, \\
&(8)\; \; (x, y)=(2, y)\; 6\geq y\geq 2, \\
&(9)\; \; (q=4)\; \; (x, y)=(1, 1),\; (1, 2),\; (1, 3),\; (2, 2),\; (2, 3), \\
&\; \; \; \; \; \; \, (q=5)\; \; (x, y)=(1, 1),\; (1, 2),\; (1, 3),\; (2, 2),\; (2, 3),\; (2, 4),\; (3, 4),\; (3, 5),\\
&\; \; \; \; \; \; \; \; \; \; \; \; \; \; \; \; \; \; \; \; \; \; \; \; \; \; \; \; \; \; \; \; \, (4, 6), \\
&(10)\; \; (x, y)=(1, 1),\; (1, 2),\; (2, 1),\; (2, 2),\; (2, 3),\; (2, 4),\; (3, 3),\; (3, 4),\; (4, 5).
\end{align*}
Next, we give below a table of $(x, y)$ such that among the $(x, y)$ satisfying $0<Q({\bf d^{{\bbR}^{+}}})\leq 1$ in each of the cases (1) to (10) given earlier, 
there exists ${\bf d}\in {\Lambda}^{J}$ satisfying the conditions of each case separation for $(x, y)$.
\begin{align*}
&(1)\; \; (x, y)=(1, 1), \\ 
&(2)\; \; (x, y)=(1, 2),\; (x, x-z)\; 0\geq z\geq -1,\; x\geq 2, \\ 
&(4)\; \; (x, y)=(2, 4),\; (3, 3),\; (3, 4),\; (3, 5),\; (4, 4),\; (4, 5),\; (4, 6), \\
&(5)\; \; (x, y)=(x, 2x)\; x\geq 3, \\
&\; \; \; \; \; \; \, (x, y)=(x, 2x-z)\; 3\geq z\geq 1,\; x\geq 4, \\ 
\end{align*}
\begin{align*}
&(6)\; \; (q=5)\; \; (x, y)=(4, 5),\; (5, 6), \\
&(7)\; \; (x, y)=(1, 3),\; (x, 2x-z)\; 2\geq z\geq -1,\; x\geq 2, \\ 
&(8)\; \; (x, y)=(2, 6), \\
&(9)\; \; (q=5)\; \; (x, y)=(3, 5),\; (4, 6), \\
&(10)\; \; (x, y)=(2, 4),\; (3, 3),\; (3, 4),\; (4, 5). 
\end{align*}
Finally, the following table gives the ${\bf d^{+}}$ of ${\bf d}\in {\Lambda}^{J}$ for which $Q({\bf d})$ is minimized for each $(x, y)$ in the above table.
\begin{align*}
&(1)\; \; (x, y)=(1, 1)\; \; {\bf d^{+}}=((1^2), (1, 1), (1, 1)),\\
&(2)\; \, (x, y)=(1, 2)\; \; {\bf d^{+}}=((1^3), (1, 2), (1^2, 1)),\\
&\; \; \; \; \; \; \, (x, y)=(x, x)\; \; {\bf d^{+}}=((x, x-1, 1), (1^{x}, x), (1^{x}, x))\; x\geq 2,\\ 
&\; \; \; \; \; \; \, (x, y)=(x, x+1)\; \; {\bf d^{+}}=((x^2, 1), (1^{x}, x+1), (1^{x+1}, x))\; x\geq 2, \\
&(4)\; \; (x, y)=(2, 4)\; \; {\bf d^{+}}=((2^3), (1^2, 4), (1^4, 2)),\\ 
&\; \; \; \; \; \; \, (x, y)=(3, 3)\; \; {\bf d^{+}}=((2^3), (2, 1, 3), (1^3, 3)),\\
&\; \; \; \; \; \; \, (x, y)=(3, 4)\; \; {\bf d^{+}}=((3, 2^2), (2, 1, 4), (1^4, 3)),\\
&\; \; \; \; \; \; \, (x, y)=(3, 5)\; \; {\bf d^{+}}=((3^2, 2), (2, 1, 5), (1^5, 3)),\\
&\; \; \; \; \; \; \, (x, y)=(4, 4)\; \; {\bf d^{+}}=((3^2, 2), (2^2, 4), (1^4, 4)),\\ 
&\; \; \; \; \; \; \, (x, y)=(4, 5)\; \; {\bf d^{+}}=((4, 3, 2), (2^2, 5), (1^5, 4)),\\ 
&\; \; \; \; \; \; \, (x, y)=(4, 6)\; \; {\bf d^{+}}=((4^2, 2), (2^2, 6), (1^6, 4)),\\ 
&(5)\; \; (x, y)=(x, 2x)\; \; {\bf d^{+}}=((x^3), (x-1, 1, 2x), (1^{2x}, x))\; x\geq 3, \\
&\; \; \; \; \; \; \, (x, y)=(x, 2x-1)\; \; {\bf d^{+}}=((x^2, x-1), (x-1, 1, 2x-1), (1^{2x-1}, x))\\
&\; \; \; \; \; \; \, x\geq 4, \\
&\; \; \; \; \; \; \, (x, y)=(x, 2x-2)\; \; {\bf d^{+}}=((x, (x-1)^2), (x-1, 1, 2x-2), (1^{2x-2}, x))\\
&\; \; \; \; \; \; \, x\geq 4, \\
&\; \; \; \; \; \; \, (x, y)=(x, 2x-3)\; \; {\bf d^{+}}=(((x-1)^3), (x-1, 1, 2x-3), (1^{2x-3}, x))\\
&\; \; \; \; \; \; \, x\geq 4, \\
&(6)\; \; (q=5)\; \; (x, y)=(4, 5)\; \; {\bf d^{+}}=((3^3), (2^2, 5), (2, 1^3, 4)),\\ 
&\; \; \; \; \; \; \, \; \; \; \; \; \; \; \; \; \; \; \; \; (x, y)=(5, 6)\; \; {\bf d^{+}}=((4^2, 3), (3, 2, 6), (2^2, 1^2, 5)),\\
&(7)\; \; (x, y)=(1, 3)\; \; {\bf d^{+}}=((1^4), (1, 3), (1^3, 1)),\\ 
&\; \; \; \; \; \; \, (x, y)=(x, 2x-2)\; \; {\bf d^{+}}=(((x-1)^3, 1), (x, 2x-2), (1^{2x-2}, x))\; x\geq 2,\\ 
&\; \; \; \; \; \; \, (x, y)=(x, 2x-1)\; \; {\bf d^{+}}=((x, (x-1)^2, 1), (x, 2x-1), (1^{2x-1}, x))\; x\geq 2,\\ 
\end{align*}
\begin{align*}
&\; \; \; \; \; \; \, (x, y)=(x, 2x)\; \; {\bf d^{+}}=((x^2, x-1, 1), (x, 2x), (1^{2x}, x))\; x\geq 2,\\ 
&\; \; \; \; \; \; \, (x, y)=(x, 2x+1)\; \; {\bf d^{+}}=((x^3, 1), (x, 2x+1), (1^{2x+1}, x))\; x\geq 2,\\ 
&(8)\; \; (x, y)=(2, 6)\; \; {\bf d^{+}}=((2^4), (2, 6), (1^6, 2)),\\ 
&(9)\; \; (q=5)\; \; (x, y)=(3, 5)\; \; {\bf d^{+}}=((2^4), (3, 5), (2, 1^3, 3)),\\ 
&\; \; \; \; \; \; \, \; \; \; \; \; \; \; \; \; \; \; \; \; (x, y)=(4, 6)\; \; {\bf d^{+}}=((3^2, 2^2), (4, 6), (2^2, 1^2, 4)),\\
&(10)\; \; (x, y)=(2, 4)\; \; {\bf d^{+}}=((1^6), (2, 4), (2^2, 2)),\\
&\; \; \; \; \; \; \, \; \, (x, y)=(3, 3)\; \; {\bf d^{+}}=((1^6), (3, 3), (2, 1, 3)),\\
&\; \; \; \; \; \; \, \; \, (x, y)=(3, 4)\; \; {\bf d^{+}}=((2, 1^5), (3, 4), (2^2, 3)),\\ 
&\; \; \; \; \; \; \, \; \, (x, y)=(4, 5)\; \; {\bf d^{+}}=((2^3, 1^3), (4, 5), (3, 2, 4)).\\
\end{align*}
Here ${\bf d^{+}}$ in the cases (1) $(x, y)=(1, 1)$, (2) $(x, y)=(1, 2)$, (7) $(x, y)=(1, 3)$, and (10) $(x, y)=(2, 4), (3, 3)$ is of finite type in $\mathscr{F}$. 
The reason for this overlap here is that the method of giving the list in \cite{mwz1999} is different from the method used here. 
A simple calculation shows that ${\bf d^{+}}$ satisfies $Q({\bf d^{+}})=1$ except for (6) $(x, y)=(5, 6)$, (9) $(x, y)=(4, 6)$, and (10) $(x,y)=(4, 5)$, so the 
left-to-right proof is now complete. 
The converse is clear from Theorem \ref{sec:3.8} .
\end{proof}

\begin{thm}
Let ${\bf (a, b, c)}\in {\Lambda}^{J}$ is of finite type in $\mathscr{J}$. 
Then there are a natural bijection between $GL_{|{\bf a}|}$-orbits in ${\rm Jl}_{\bf (a, b, c)}(V)$ and families $M=(m_{\bf d})$ of nonnegative integers satisfying 
\[
\sum_{\bf d} (m_{\bf d}){\bf d}={\bf (a, b, c)}
\]
and indexed by ${\bf d}\in {\Lambda}^{J}$, 
where ${\bf d}$ is of finite type in $\mathscr{J}$ and $Q({\bf d})=1$.
\end{thm}
\begin{proof}
It is clear from Proposition \ref{sec:2.1}  and Proposition \ref{sec:3.5}  and Proposition \ref{sec:3.6}.
\end{proof}

\section{Double Flag Varieties}
Now introduce the following notation:\\
For ${\bf d=(a, b, c)=(a, (b'}, b_{p}), ({\bf c'}, c_{q}))\in {\Lambda}^{J}$, 
\[
{\rm Dl}_{\bf d}(V_2\bigoplus V_3):={\rm Fl}_{\bf a}(V_1)\times {\rm Fl}_{\bf b'}(V_2)\times {\rm Fl}_{\bf c'}(V_3), 
\]
 where $\dim{V_2}=c_{q}$, $\dim{V_3}=b_{p}$, and $V_1=V_2\bigoplus V_3$.\\
${\rm Dl}_{\bf d}(V_2\bigoplus V_3)$ is called the Double flag variety.

\begin{definition}
A dimension vector ${\bf d=(a, b, c)}\in {\Lambda}^{J}$ is called finite type in $\mathscr{D}$ if the number of $GL(V_2)\times GL(V_3)$-orbits in ${\rm Dl}_{\bf d}(V_2\bigoplus V_3)$ is finite.
\end{definition}

\begin{definition}
Define the {\bf Double flag category} ${\mathscr{D}}_{r, p, q}$ as follows$:$ \\
${\mathscr{D}}_{r, p, q}$ is a full sub-category of ${\mathscr{J}}_{r, p, q}$, whose objects are given in the following forms of $(V_2\bigoplus V_3, A', B', C')\in {\mathscr{J}}_{r, p, q}$$:$\\
\[
\xymatrix@=10pt{
 & & & & V_2 \ar[dl]_{\begin{pmatrix} 1\\ 0\\ \end{pmatrix} } & B_{p-1} \ar[l] & \cdots \ar[l] & B_1 \ar[l] \\
 A_1  \ar[r]  & \cdots \ar[r] & A_{r-1} \ar[r] & V_1=V_2\bigoplus V_3  \\
 & & & & V_3 \ar[ul]^{\begin{pmatrix} 0\\ 1\\ \end{pmatrix} } & C_{q-1} \ar[l] & \cdots \ar[l] & C_1 \ar[l]
}
\]
 where $(A, B, C)$ is a triple of flag in $(V_1, V_2, V_3)$ belong to ${\rm Dl}_{\bf d}(V_2\bigoplus V_3)$ 
for some ${\bf d}:=({\bf a}, ({\bf b}, \dim{V_3} ), ({\bf c}, \dim{V_2} ))\in {\Lambda}^{J}$. 
\end{definition}
When there is no risk of ambiguity, we drop the subscripts and write ${\mathscr{D}}$.\\

Then by embedding ${\mathscr{D}}_{r, p, q}$ in $\rep{\k Q_{r, p, q}}$, 
the object of ${\mathscr{D}}_{r, p, q}$ is directly decomposed on $\rep{\k Q_{r, p, q}}$ except for the order and the isomorphism, 
by applying Krull-Schmidt Theorem\cite{krull}. This decomposition is also a natural decomposition on ${\mathscr{D}}_{r, p, q}$.

Then the following are equivalent for $(A, B, C), (A', B', C')\in {\rm Dl}_{\bf d}(V_2\bigoplus V_3)$: 
\begin{align*}
&(1)\; (GL(V_2)\times GL(V_3))(A, B, C)=(GL(V_2)\times GL(V_3))(A', B', C') \\
&(2)\; (V_2\bigoplus V_3, A, B, C)\cong (V_2\bigoplus V_3, A', B', C')\; on\; {\mathscr{D}}
\end{align*}


Therefore we get the following :\\
(1) $\# \{ GL(V_2)\times GL(V_3)(A, B, C)\, |\, (A, B, C)\in {\rm Dl}_{\bf d}(V_2\bigoplus V_3)\} =\# \{ $isomorphism classes of $F_{\bf d}\in {\mathscr{D}}\} $.\\
(2) A concrete description of the orbit decomposition of ${\rm Dl}_{\bf d}(V_2\bigoplus V_3)$ can be given by finding all the isomorphism classes of $F_{\bf d}$. \\

Now we can translate the problem of the number and description of $GL(V_2)\times GL(V_3)$-orbits of ${\rm Dl}_{\bf d}(V_2\bigoplus V_3)$ into the problem of finding the isomorphism classes of the representation of the quiver.

\begin{prop}\label{sec:4.3} 
${\mathscr{J}}_{r, p, q}$ and ${\mathscr{D}}_{r, p, q}$ are equivalent categories.
\end{prop}
\begin{proof}
Define functor $F:{\mathscr{J}}_{r, p, q} \longrightarrow {\mathscr{D}}_{r, p, q}$ as follows, 
for the following objects \\
\[
\xymatrix{
& & & & & B_{p-1} \ar[dl]_{{\alpha}_{p-1}} & \cdots \ar[l] & B_1 \ar[l] \\
J : & A_1  \ar[r]  & \cdots \ar[r] & A_{r-1} \ar[r] & V  \\
& & & & & C_{q-1} \ar[ul]^{{\beta}_{q-1}} & \cdots \ar[l] & C_1 \ar[l] 
}
\]
of ${\mathscr{J}}_{r, p, q}$. 
Functor $F$ transfers objects and morphisms to the branches $B$ and $C$ only, without changing the branch of $A$, as follows: \\
\[
\xymatrix{
B : & V \ar[d]^f & B_{p-1} \ar[l]_{{\alpha}_{p-1}} \ar[d]^{f_{p-1}} & B_{p-2} \ar[l]_{{\alpha}_{p-2}} \ar[d]^{f_{p-2}} & \cdots \ar[l] & B_1 \ar[l]_{{\alpha}_{1}} \ar[d]^{f_{1}} \\
B' : & V' & B'_{p-1} \ar[l]_{{{\alpha}'}_{p-1}} & B'_{p-2} \ar[l]_{{{\alpha}'}_{p-2}} \ar[d]^{F} & \cdots \ar[l] & B'_1 \ar[l]_{{{\alpha}'}_{1}} \\
F(B) : & \Im{{\alpha}_{p-1}}\bigoplus \Im{{\beta}_{q-1}} \ar[d]^f & \Im{{\alpha}_{p-1}} \ar[l]_{\begin{pmatrix} 1\\ 0\\ \end{pmatrix} } \ar[d]^{f|_{\Im{{\alpha}_{p-1}}}} & B_{p-2} \ar[l]_{{\alpha}_{p-1}{\alpha}_{p-2}} \ar[d]^{f_{p-2}} & \cdots \ar[l] & B_1 \ar[l]_{{\alpha}_{1}} \ar[d]^{f_{1}} \\
F(B') : & \Im{{{\alpha}'}_{p-1}}\bigoplus \Im{{{\beta}'}_{q-1}} & \Im{{{\alpha}'}_{p-1}} \ar[l]^{\begin{pmatrix} 1\\ 0\\ \end{pmatrix} } & B'_{p-2} \ar[l]_{{{\alpha}'}_{p-1}{{\alpha}'}_{p-2}} & \cdots \ar[l] & B'_1. \ar[l]_{{{\alpha}'}_{1}} 
}
\]
Similarly for $C$.\\
Then $J$ and $F(J)$ are isomorphic as representations of a quiver, and it is easy to see that $F$ is an equivalence of categories. 
\end{proof}

Therefore, from Proposition \ref{sec:4.3} , we immediately obtain the following Theorem \ref{sec:4.4}  and Theorem \ref{sec:4.5}  by appropriating the result of section 2.

\begin{thm}\label{sec:4.4}
${\bf d}\in {\Lambda}^{J}$ is of finite type in $\mathscr{D}$ if and only if ${\bf d}\in {\Lambda}^{J}$ is called finite type in $\mathscr{J}$. 
\end{thm}

\begin{thm}\label{sec:4.5}
Let ${\bf (a, b, c)}\in {\Lambda}^{J}$ is of finite type in $\mathscr{D}$. 
Then there are a natural bijection between $GL(V_2)\times GL(V_3)$-orbits in ${\rm Dl}_{\bf (a, b, c)}(V_2\bigoplus V_3)$ and families $M=(m_{\bf d})$ of nonnegative integers satisfying 
\[
\sum_{\bf d} (m_{\bf d}){\bf d}={\bf (a, b, c)}
\]
and indexed by ${\bf d}\in {\Lambda}^{J}$, 
where ${\bf d}$ is of finite type in $\mathscr{D}$ and $Q({\bf d})=1$.
\end{thm}





\section*{Acknowledgments}
The author thanks Hiroyuki Ochiai for helpful references and discussions.









\end{document}